\documentclass[final]{aptpub}
\usepackage{amsmath,amssymb,amstext,color}
\usepackage{graphicx}
\usepackage{personal}
\newcommand{\floor}[1]{\lfloor {#1} \rfloor}
\newcommand{\eps}{\varepsilon}
\numberwithin{equation}{section}
\numberwithin{theorem}{section}
\numberwithin{proposition}{section}
\numberwithin{lemma}{section}
\numberwithin{corollary}{section}
\authornames{Harsha Honnappa} % insert the authors here for use in running head
\shorttitle{Rare Events of Transitory Queues} % insert short title here for use in running head

\begin{document}

\title{Rare Events of Transitory Queues} % insert title - use \\ if it requires more than one line.

\authorone[School of Industrial Engineering, Purdue University]{Harsha Honnappa} % Affiliation is just the name of your university or institution

\addressone{315 N. Grant St., West Lafayette IN 47906. honnappa@purdue.edu} % Your postal address goes here.

 % insert keywords separated by a semicolon

 % insert the primary Maths Subject Classification number in the first bracket
         % and the secondary ams number(s) in the second bracket
         % e.g. \ams{60E20}{49G03;49F10}

% \begin{frontmatter}
% \title{Rare Events of a Transitory Queue}
% \author[puie]{
% 	Harsha Honnappa\corref{cor1}\fnref{fn1}
% }
% \ead{honnappa@purdue.edu}
% \address[puie]{
% 	School of Industrial Engineering, Purdue University, West
%         Lafayette IN 47906.	
% }
% \begin{keyword}
% workload, empirical process, large deviations, Varadhan's Lemma, population acceleration
% \end{keyword}

\begin{abstract}
We study the rare event behavior of the workload process in a transitory queue,
where the arrival epochs (or `points') of a finite number of jobs are assumed to be the ordered statistics of 
independent and identically distributed (i.i.d.) random variables. The service
times (or `marks') of the jobs are assumed to
be i.i.d. random variables with a general distribution, that are jointly
independent of the arrival epochs. Under the assumption that the
service times are strictly positive, we derive the large
deviations principle (LDP) satisfied by the workload process. The
analysis leverages the connection between ordered statistics and
self-normalized sums of exponential random variables to establish the LDP. This paper
presents the first analysis of rare events in transitory queueing
models, supplementing prior work that has focused on fluid and
diffusion approximations.
\end{abstract}

%\maketitle

\section{Introduction}
We explicate the rare event behavior of a `transitory' queueing model, by proving a large
deviations principle (LDP) satisfied by the workload process of the queue. A formal definition of a transitory queue follows
from \cite{HoJaWa2016}:

\begin{definition} [Transitory Queue]
  Let $A(t)$ represent the cumulative traffic entering a queueing
  system. The queue is transitory if $A(t)$ satisfies
  \begin{equation}
    \label{eq:1}
    \lim_{t \to\infty} A(t) < \infty ~\text{a.s.}
  \end{equation}
\end{definition}

\noindent 
We consider a specific transitory queueing model where the arrival epochs
(or `points')
of a finite but
large number of jobs, say $n$, are `randomly
scattered' over $[0,\infty)$; that is the arrival epochs $(T_1,\ldots,T_n)$ are i.i.d.,
and drawn from some distribution with support in $[0,\infty)$. We
assume that the service times (or `marks') $\{\nu_1,\ldots, \nu_n\}$ are i.i.d., 
jointly independent of the
arrival epochs and with log moment generating function that satisfies
$\varphi(\theta) < \infty$ for $\theta \in \bbR$. We call this
queueing model the $RS/GI/1$ queue (`$RS$' standing for
\textit{randomly scattered}; this was previously dubbed the $\D_{(i)}/GI/1$ queueing model in \cite{HoJaWa2014}). % Our objective is to characterize the rare events
% of the workload process in a high intensity/large population
% regime. To that end, we also
% assume that the moment generating function of the service time satisfies
% $\varphi(\theta) < \infty$ for $\theta \in \bbR$.

While the i.i.d. assumption on the arrival epochs implies that this is a
homogeneous model, \cite{HoJaWa2014} shows that the workload process displays time-dependencies in
the large population fluid and diffusion scales, that mirrors 
those observed for `dynamic rate' queueing models where
time-dependent arrival rates are explicitly
assumed. This indicates that the rare event behavior of the workload
or queue length process should be atypical compared to that of time-homogeneous queueing models (such as the $G/G/1$
queue; see \cite{GaOcWi2004}). Further, while the standard dynamic rate traffic model
is a nonhomogeneous Poisson process that necessarily has
independent increments, it is less than obvious that is a reasonable
assumption for many service systems. \cite{Gl2012,Gl2014}, for
instance, highlights data analysis and simulation results in the call center
context that indicate that independent increments might not be appropriate. A tractable
alternative is to assume that the increments
are exchangeable \cite{Al1985}. Lemma~10.9 in \cite{Al1985} implies
that \textit{any} traffic process over the horizon $[0,1]$ with exchangeable
increments is necessarily equal in distribution to the empirical sum process
\begin{equation} \label{eq:empiric}
\sum_{i=1}^N \mathbf 1_{\{T_i \leq t\}}~\forall t \in [0,1],
\end{equation}
where the $\{T_i,~1\leq i \leq N\}$ are independent and uniformly
distributed in [0,1]. In \cite{HoJaWa2014} we defined
\eqref{eq:empiric} as the traffic count process for the
$RS/GI/1$ queue. This can be considered the canonical model of a
transitory traffic process with exchangeable increments. Thus, the
results in this paper can also be viewed as explicating the rare events
behavior of queueing models with exchangeable increments. To the best
of our knowledge this has not been reported in the literature before.

Transitory queueing models, and the $RS/GI/1$ queue in
particular, have received some recent interest in the applied
probability literature, besides \cite{HoJaWa2014}. In forthcoming work~\cite{HoGl2016} studies large deviations and
diffusion approximations to the workload process in a `near balanced'
condition on the offered load to the system. The current paper
complements this by not assuming the near balanced
condition. In recent work,
\cite{BeHoJvL2015} established diffusion approximations to the queue
length process of the $\D_{(i)}/GI/1$ queue under a uniform
acceleration scaling regime, where in it is assumed that the ``initial
load'' near time zero satisfies $\rho_n = 1 +
\beta n^{-1/3}$. This, of course, contrasts with the population
acceleration regime considered in this paper, where the offered load
is accelerated by the population size at all time instances in the
horizon $[0,1]$. The same authors have also considered the
effect of heavy-tailed service in transitory queues in
\cite{BeHoJvL2016}, and established weak limits to the scaled workload
and queue length processes to reflected alpha-stable processes.

In the ensuing discussion, we will largely focus on the case that the arrival epochs are
uniformly distributed with support $[0,1]$. Our first result in Theorem
\ref{thm:cramer-order} establishes a large deviations result for the
ordered statistics process
\(
(T^n(t) := T_{(\lfloor nt \rfloor)}~\forall t \in [0,1])~\text{and}~n\geq 1,
\)
where $T_{(j)}$ represents the $j$th order statistic. % We show that for any small level $a \in [0,t)$, the large deviations rate
% function is $-t \log\left(t\frac{1-a}{a}\right) - (1-t) \log \left(1-t
%   + \frac{a}{1-a}\right)$. 
This result parallels that in
\cite{DuMaTo2010}, where the authors derive a sample path large
deviations result for the ordered statistics of i.i.d. uniform random variables. Our results deviate from this result in a couple of
ways. First, we do not require a full sample path
LDP, since we are interested in understanding the large deviations
of the workload at a given point in time. Second, our proof technique
is different and explicitly uses
the connection between ordered statistics and self-normalized sums of
exponential random variables.  It is also important to note the result in \cite{Bo2007}, where the
author uses Sanov's theorem to prove the large deviation principle for
$L$-statistics, which could be leveraged to establish the LDP for the
traffic process in~\eqref{eq:empiric} and, hence, the number-in-system
process. The objective of our study, on the other hand, is the workload process. In Corollary~\ref{cor:cramer-order} we use the
contraction principle to extend this large deviations result to
arrival epochs that have distribution $F$ with positive support, under
the assumption that the distribution is absolutely
continuous and strictly increasing. However, much of the `heavy-lifting' for the workload LDP
can be demonstrated with uniform order statistics arrival epochs, so
in the remainder of the paper we do not emphasize the extension to
more generally distributed arrival epochs.

In Proposition \ref{thm:marked-order} we make use of the
proof of Theorem \ref{thm:cramer-order} and the well known Cramer's
Theorem \cite[Theorem 2.2.3]{DeZe2010} to derive the large
deviation rate function for the offered load process 
\(
X^n(t) := S^n(t) - T^n(t) ~\forall t\in[0,1]~\text{and}~n\geq 1,
\)
where 
\(
S^n(t) := \sum_{i=1}^{\floor{nt}} \nu_i
\)
is the partial sum of the service times.
 % by leveraging the results in \cite{LySe1987} to establish a joint LDP
 % for $\{(S^n(t),T^n(t)),~n\geq 1\}$. 
 Interestingly enough, the LDP (and
 the corresponding good rate function) shows that the most likely path
 to a large deviation event depends crucially on both the sample path
 of the offered load process up to $t$ as well as the path
 \textit{after} $t$. This is a direct reflection of the fact that the
 traffic process is exchangeable and that there is long-range
 dependence between the inter-arrival times (which are `spacings'
 between ordered statistics, and thus finitely exchangeable). 

We prove the LDP for the workload process \[W^n(t) :=
\Gamma(X^n)(t) = \sup_{0 \leq s \leq t} (X^n(t) - X^n(s)),\] for fixed
$t \in [0,1]$, by exploiting the
continuity of the reflection regulator map $\Gamma(\cdot)$. However, to do so, we
first establish two auxiliary results: in
Proposition~\ref{prop:exponential-eq} we prove the exponential
equivalence of the workload process and
a linearly interpolated version $\tilde X^n$. Then, in
Proposition~\ref{prop:sample-path-offered} we prove the
LDP satisfied by the `partial' sample paths $(\tilde X^n(s),~0\leq s\leq t)$ of the offered load process for 
fixed $t \in [0,1]$. Then, in
Theorem~\ref{thm:ldp-workload} we establish the LDP for the workload
process by applying the contraction mapping theorem with the reflection
regulator map and exploiting the two propositions mentioned above.  We
conclude the paper with a summary
and comments on future directions for this research.

% \noindent \textbf{Organization}
% ~Finally, a note on the organization of the paper. We start by describing our
% model in Section~\ref{sec:model}, paying particular attention to the
% definition of the workload process in terms of the offered load
% $X^n$. Next, in Section~\ref{sec:offered-load} we prove our first result establishing an LDP for the
% ordered statistics process in Theorem~\ref{thm:cramer-order} by
% exploiting Proposition~\ref{prop:uni-expo}. This serves as a precursor to the full LDP
% for the offered load in Theorem~\ref{thm:marked-order}, illustrating
% how Proposition~\ref{prop:varadhan} can be utilized. Next, we
% establish the LDP for the workload in Section~\ref{sec:workload}, as a
% consequence of Theorem~\ref{thm:marked-order}.

% Next, using the results for the i.i.d. uniform arrival epoch case, we
% consider a more general situation where the arrival epochs are
% distributed over $[0,\infty)$ following some distribution $F$. We
% extract the corresponding rate function using the quantile
% transformation of the arrival time distribution to i.i.d. uniform
% random variables.
\subsection{Notation}
We assume that all random elements are defined with respect to an
underlying probability sample space $(\Omega, \sF, \bbP)$. % Let
% $(\sU,\sB, m)$ represent the sample space on the unit interval $\sU
% :=[0,1]$, where $\sB$ is the set of all Borel sets in $\sU$ and $m$ is
% the Lebesgue measure.
We denote convergence in probability by
$\stackrel{P}{\to}$. We denote the space $\sX$ and topology of
convergence $\sT$ by the pair $(\sX,\sT)$, where appropriate. In
particular we note $(\sC[0,t],\sU)$, the space of
continuous functions with domain $[0,t]$, equipped with the uniform
topology. We also designate $\bar \sC[0,t]$ as the space of all
continuous functions that are non-decreasing on the domain
$[0,t]$. $\|\cdot\| = \sup_{0\leq s \leq 1} (\cdot)$ represents the
supremum norm on $\sC[0,1]$. Finally, we will use the following standard definitions in
the ensuing results:
\newpage
\begin{definition} [Rate Function]
 Let $\sX$ be Hausdorff topological space. Then,
  \begin{itemize}
  \item a rate function is a lower semicontinuous mapping $I : \sX \to
    [0,\infty]$; i.e., the level set $\{x \in \sX : I(x) \leq \a\}$
    for any $\a \in [0,\infty)$ is a closed subset of $\sX$, and
  \item a rate function is `good' if the level sets are also compact.
  \end{itemize}
\end{definition}

\begin{definition}[Large Deviations Principle (LDP)]
  The sequence of random elements $\{X_n,~n \geq 1\}$ taking values in
  the Hausdorff topological space $\sX$ satisfies a
  large deviations principle (LDP) with rate function $I : \sX \to \bbR$ if

\noindent a) for each open set $G \subset \sX$
\[
\liminf_{n \to \infty} \frac{1}{n} \log \bbP(X_n \in G) \geq -\inf_{x
  \in G} I(x), ~\text{and}
\] 

\noindent b) for each closed set $F \subset \sX$
\[
\limsup_{n \to \infty} \frac{1}{n} \log \bbP (X_n \in F) \leq -\inf_{x
\in F} I(x).
\]

\end{definition}

\begin{definition}[Weak LDP]
  The sequence of random elements $\{X_n,~n \geq 1\}$ taking values in
  the Hausdorff topological space $\sX$ satisfies a weak
  large deviation principle (WLDP) with rate function $I$ if 

\noindent a) for each open set $G \subset \sX$
\[
\liminf_{n \to \infty} \frac{1}{n} \log \bbP(X_n \in G) \geq -\inf_{x
  \in G} I(x),~\text{and}
\]

\noindent b) for each compact set $K \subset \bbR$,
\[
\limsup_{n\to \infty} \frac{1}{n} \log \bbP(X_n \in K) \leq -\inf_{x
  \in K} I(x).
\]
\end{definition}

\begin{definition}[LD Tight]
 \label{def:ld-tight}
  A sequence of random elements $\{X_n,~n\geq 1\}$ taking values in
  the Hausdorff topological space $\sX$ is large deviation
  (LD) tight if for each $M < \infty$, there exists a compact set
  $K_M$ such that
  \[
  \limsup_{n \to \infty} \frac{1}{n} \log \bbP(X_n \in K_M^c) \leq -M.
  \]
\end{definition}
\section{Model}\label{sec:model}
Let $\{(T_{(i)},\nu_i),~i = 1,2,\ldots,n\}$ for $n \in \bbN$ represent
a marked finite point process, where $\{T_{(i)}, ~i=1,2,\ldots,n\}$
are the epochs of the point process and $\{\nu_i, ~i=1,2,\ldots,n\}$
are the marks. We assume that the two sequences are independent of
each other. $\{T_{(i)}, ~i=1,2,\ldots,n\}$ are the order statistics of
$n$ independent and identically distributed (i.i.d.) random variables with
support $[0,\infty)$ and absolutely continuous distribution $F$. $\{\nu_i,~i=1,2,\ldots,n\}$ are i.i.d. random
variables with support $[0,\infty)$, cumulant generating function
$\varphi(\theta)<\infty$ for some $\theta \in \bbR$ and mean $\bbE[\nu_1] =
1/\mu$. We will also assume that $\bbP(\nu_1 > 0) = 1$, for technical
reasons. Let $\bbD := \{\theta \in \bbR : \varphi(\theta) < \infty\}$
and we assume $0 \in \bbD$. In relation to the
queue, $(T_{(j)},\nu_j)$ represents the arrival epoch and service
requirement of job $j$, and $n$ is the total arrival population. It is
useful to think of the $n$ marked points, or $(T_{(i)},\nu_i)$ pair, being `scattered' over
the horizon following the distribution $F$.

Let $\{\nu_i^n:=\nu_i/n, ~i=1,2,\ldots,n\}$ be a `population accelerated'
sequence of marks. Assume that $(T_{(0)},\nu_0^n) = (0,0)$. The (accelerated) workload ahead of the $j$th job is
\(
W^n_j = (W^n_{j-1} + \nu^n_{j-1} - (T_{(j)} - T_{(j-1)}))_+,
\)
where $(\cdot)_+ := \max\{0,\cdot\}$. By unraveling the recursion, and under the assumption that the queue starts empty, it can be shown that
\[
W^n_j \stackrel{D}{=} (S^n_{j-1} - T_{(j)}) + \max_{0 \leq i \leq j-1} \left( -(S^n_i - T_{(i+1)})\right),
\]
where $S^n_{j-1} := \sum_{i=0}^{j-1} \nu_i^n$. We define the workload process as 
\(
(W^n(t),~t\in[0,1]) := (W^n_{\lfloor nt \rfloor},~t\in[0,1]).
\)
Using the unraveled recursion it can be argued that
\begin{eqnarray}
\label{eq:workload-approx}
W^n(t) \stackrel{D}{=} X^n(t) + \max_{0 \leq s \leq t} (-X^n(s))
\end{eqnarray}
where $X^n(t) :=
n^{-1}\sum_{i=0}^{\floor{nt}}  \nu_i - T_{(\floor{nt})} = S^n(t) - T^n(t)$ (where
$T_{(0)} = 0$) for $t \in [0,1]$ is the offered load process, under
the assumption that $S_0^n = 0$ (i.e., the queue starts empty). Thus, it suffices
to study $\Gamma(X^n)(t):=X^n(t) + \max_{0 \leq s \leq t}(-X^n(s))$,
where $\Gamma : \sD[0,1] \to \sD[0,1]$ is the so-called Skorokhod
regulator map. For future reference, we
call $(T^n(t),~t \in [0,1]) := (T_{(\floor{nt})},~t \in [0,1])$ as
the ordered statistics process.

% Proposition 3 in \cite{HoJaWa2014} proves a workload fluid limit,
% which we include here for completeness. Let $A^n(t) = \sum_{i=1}^n \mathbf{1}_{\{T_i \leq t\}}$ be the arrival counting process, where $T_i$ are i.i.d. with distribution $F$. Let $M : [0,\infty) \to [0,\infty)$ such that $M(t) = \mu t$ for any $t \geq 0$, where $E[\nu_i] = 1/\mu$ is the average service time. 

% \begin{prop}
% 	Let $\tilde W^n(t) = \sum_{i=0}^{A^n(t)} \nu^n_i - t + I^n(t)$
%         for all $t \in [0,\infty)$, where $I^n(t)$ is the cumulative idle time of the queue. Then,
% 	\[
% 		\tilde W^n \stackrel{a.s.}{\to} \bar W := \frac{\Gamma(F - M)}{\mu} ~ \text{in}~(\sD,U)
% 	\]
% 	as $n \to \infty$, where $\Gamma(x) = x + \sup_{0 \leq s \leq t}(-x(s))$.
% \end{prop}

% Lemma~\ref{lem:equivalence} below shows that $W^n(t)$ and $\tilde
% W^n(t)$ (for fixed $t \in [0,1]$) are exponentially
% equivalent, implying that $ W^n(t) \stackrel{P}{\to} \bar W(t)$ as $n \to
% \infty$.

% %  where $\stackrel{P}{\to}$ indicates convergence in
% % probability. 

% \begin{lemma}\label{lem:equivalence}
%   Fix $t \in [0,1]$. Then, for any $\epsilon > 0$, it follows that
%   \begin{equation}
%     \label{eq:equivalence}
%     \limsup_{n\to\infty} \frac{1}{n} \log \bbP\left( \left| W^n (t) -
%         \tilde W^n(t)\right| > \epsilon \right) = -\infty.
%   \end{equation}
% \end{lemma}

% \begin{proof}
%   $\QED$
% \end{proof}

We propose to study the workload process in the large population
limit and, in particular, understand the rare event behavior in this
limit. As a precursor to this analysis, it is useful to consider
what a ``normal deviation'' event for this process would be. In
particular, The next
proposition proves a functional strong law of large numbers (FSLLN) result for the
workload process, that exposes the first order behavior
of the workload sample path, in the large population limit.

\begin{proposition}
  The workload process $W^n$ satisfies
  \[
  W^n \to \bar W = \frac{1}{\mu}\Gamma\left( F - M\right)~\text{in}~(\sC[0,1],\sU)~\text{a.s.}
  \]
as $n \to \infty$, where $M(t) = \mu t$.
\end{proposition}
\begin{proof}
  First assume that $\{0 < T_{(1)} \leq \ldots \leq T_{(n)} < 1\}$ are the ordered
  statistics of $n$ i.i.d. uniform random
  variables. Then, by \cite[Lemma 5.8]{ChYa01} it follows that the
  ordered statistics process satisfies
  \(
  (T^n,S^n) \to (e, \m^{-1} e) ~\text{in}~(\sC[0,1],\sU)~\text{a.s. as}~ n \to \infty,
  \)
  where $e : \bbR \to \bbR$ is the identity map, and the joint
  convergence follows due to the fact that the arrival epochs and
  service times are independent sequences. Let $\bar X := (\m^{-1} e -
  e)$, which is continuous by definition. Since subtraction is
  continuous under the uniform metric topology it follows that 
  \(
  X^n \to \bar X := \m^{-1}(e - M) ~\text{in}~(\sC[0,1],\sU)~\text{a.s. as}~ n \to \infty.
  \)
  Finally, since
  $\Gamma(\cdot)$ is continuous under the uniform metric, and the limit
  function $\Gamma(\bar X)$ is continuous, it follows that $W^n \to
  \bar W$ in $(\sC[0,1],\sU)$ a.s. as $n \to \infty$. The limit result
  for generally distributed arrival epochs follows by an
  application of the quantile transform to the arrival epochs.
$\QED$
\end{proof}

From an operational perspective it is
useful to understand the likelihood that the workload exceeds an
abnormally large threshold. More precisely, we are interested in the likelihood that
for a given $t \in [0,1]$ $W^n(t) > w$, where $w >> \bar W(t)$. While
this is quite difficult to prove for a fixed $n$, we prove an LDP for
the workload process as the population size $n$ scales to infinity,
that will automatically provide an approximation to the likelihood of
this event. In the ensuing exposition, we will largely focus on the
analysis of a queue where the arrival epochs are modeled as the
ordered statistics of i.i.d. uniform random variables on
$[0,1]$. However, the results can be straightforwardly extended to a
more general case where the arrival epochs have distribution $F$
(with positive support), that is absolutely
continuous with respect to the uniform distribution. 
% In this paper, we take a first step towards understanding this by studying the large deviations behavior of $X^n(t)$ in the large population acceleration regime.

Our agenda for proving the workload LDP will proceed in several steps. First, we prove an LDP for the
ordered statistics process of i.i.d. uniform random variables. The proof
of this result will then
be used to establish an LDP for  the offered load process
$X^n(t)$. Next, we use a projective limit to establish the LDP for the
sample path of the offered load process $(X^n(s),~0\leq s\leq t)$, for
each fixed $t \in [0,1]$. Finally, we prove an LDP for the workload
process by applying the contraction principle to the LDP for the
sample path $(X^n(s),~0\leq s\leq t)$, transformed through the
Skorokhod regulator map $\Gamma(\cdot)$.

\section{An LDP for the offered load}\label{sec:offered-load}
%Notice that the large deviations for the offered load 
% $X^n(t)$ depends on the large deviations of
% $T^n(t)$.
\subsection{LDP for the ordered statistics process}
As a precursor to the LDP for the
offered load process we prove one for the ordered statistics
process $(T^n(t),~t \in [0,1]) := (T_{(\floor{nt})},~t \in[0,1])$ by leveraging the following % well-known fact that 
% \begin{equation} \label{eq:order-expo}
% T_{(i)} \stackrel{D}{=} Z_i/Z_{n+1},
% \end{equation} 
% where $Z_i := \sum_{j=1}^i \xi_j$ and $\{\xi_j, ~j=1,2,\ldots,n+1\}$
% are i.i.d. unit mean exponential random variables, which allows us to
% recast the problem in terms of partial sums of positive random variables. This, in turn, permits us to leverage the rich theory of large deviations of random walks to establish the appropriate LDP.
well-known relation between the order statistics of uniform random
variables and partial sums of unit mean exponential random variables:

\begin{prop} \label{prop:uni-expo}
	Let $0 < T_{(1)}< T_{(2)}<\cdots<T_{(n)} < 1$ be the ordered
        statistics of independent and uniformly distributed random
        variables, and $\{\xi_j, ~1 \leq j \leq n+1\}$ independent
        mean one exponential random variables. 
	Then,
	\begin{equation}
          \label{eq:order-expo}
	\{T_{(j)}, ~1 \leq j \leq n\} \stackrel{D}{=} \left\{\frac{Z_j}{Z_{n+1}}, ~1 \leq j \leq n \right\}, 
	\end{equation}
	where $Z_{j} := \sum_{i=1}^{j} \xi_i$.
\end{prop}

Proofs of this result can be found in \cite[Lemma 8.9.1]{Du10}. % Our first result proves an LDP satisfied by the ordered
% statistic process, $(T^n(t),~t \in [0,1]) := (T_{\floor{nt}},~t \in[0,1])$. 
% and serves to demonstrate how Proposition~\ref{prop:uni-expo} will be used to establish a large deviations result for the offered
 % load process (and consequently, the workload).
Now, consider the convex, continuous function $I_t : [0,1] \to \bbR$ indexed by $t \in [0,1]$,
 \begin{equation}
   \label{eq:rate-function-OS}
   I_t(x) = t \log\left( \frac{t}{x} \right) + (1-t) \log \left(\frac{1-t}{1-x} \right).
 \end{equation}
Figure~\ref{fig:rate-function-OS} depicts~\eqref{eq:rate-function-OS}
for different index values $t \in [0,1]$. In
Theorem~\ref{thm:cramer-order} below we show that $I_t$ is the good
rate function of the LDP satisfied by the ordered statistics
process. It is interesting to note that this function is also the rate
function satisfied by a sequence of i.i.d. Bernoulli random variables
with parameter $t$; see the citations in \cite{DiZa1992}. 

\begin{figure}[h]
  \centering
  \includegraphics[scale=0.5]{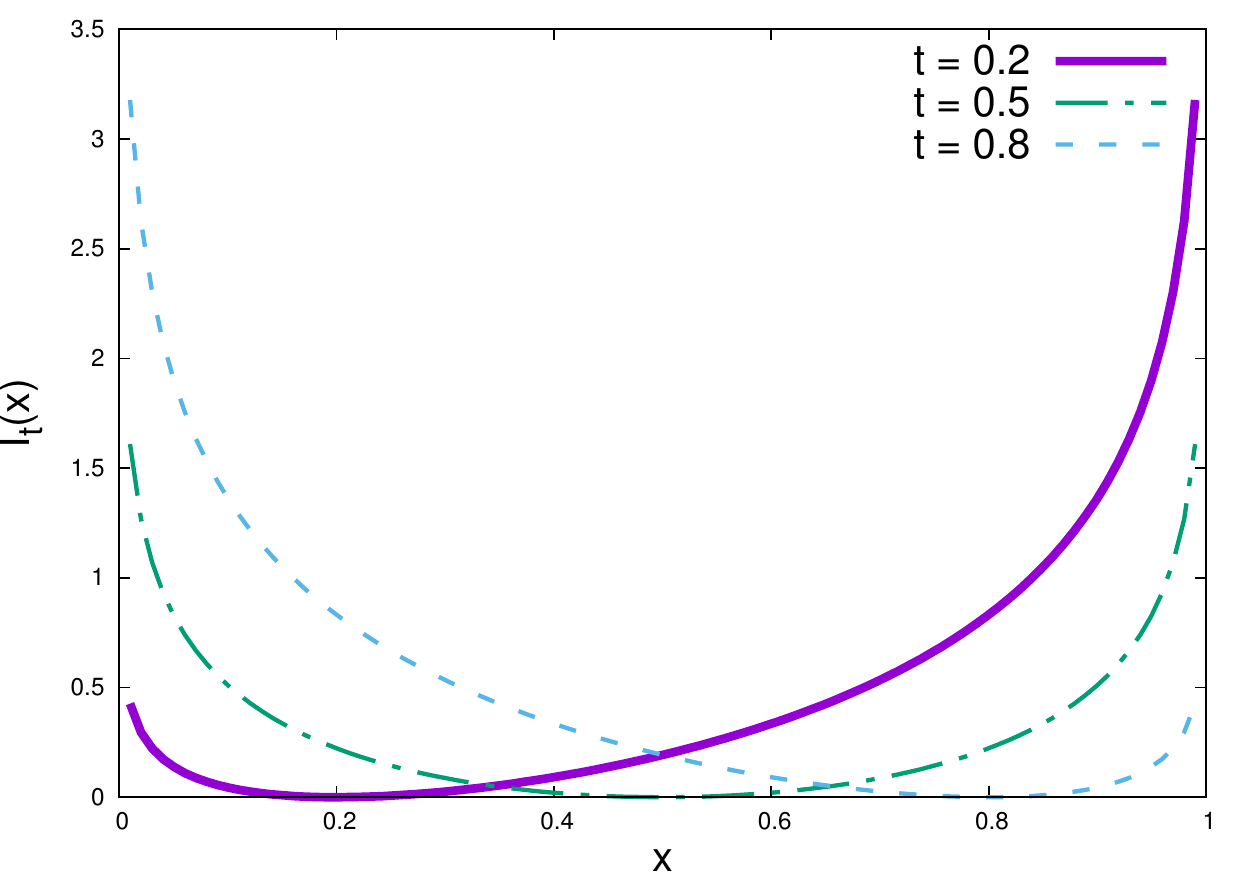}
  \caption{Rate function for the ordered statistics process.}
  \label{fig:rate-function-OS}
\end{figure}

\begin{theorem} [LDP for the Ordered Statistics Process]
	\label{thm:cramer-order}
        Fix $t \in [0,1]$. The ordered statistics process $T_{n}(t)$
        satisfies the LDP with good rate function~\eqref{eq:rate-function-OS}.
%         \noindent a) For any closed set $F \subset [0,1]$,
% \[
% \limsup_{n \to \infty} \frac{1}{n} \log \bbP \left( T^n(t) \in F
% \right) \leq -\inf_{x \in F} I_t(x), ~\text{and}
% \]

% \noindent b) for any open set $G \in [0,1]$,
% \[
% \liminf_{n \to \infty} \frac{1}{n} \log \bbP \left( T^n(t) \in G
% \right) \geq - \inf_{x \in G} I_t(x).
% \]
\end{theorem}
\begin{proof}

\noindent a) Let $F \subset [0,1]$ be closed. There are two cases to
consider. First, if $t \in F$, then $I_t(F) := \inf_{x \in F} I_t(x) =
0$, by definition. Thus, we assume that $t \not \in F$. Let $x_+ := \inf \{x \in F : x > t\}$ and $x_{-} := \sup \{x \in F : x
< t\}$. If $\sup F < t$ then we define $x_+ = 1$, and if $\inf F > t$
we set $x_- = 0$. Since $t \not \in F$, there exists a connected open
set $F^c \supseteq (x_-, x_+) \ni t$. 

Now, let $0 \leq a < t$. Proposition \ref{prop:uni-expo} implies that
	\begin{eqnarray}
          \nonumber
	\bbP(T_{(\lfloor nt \rfloor)} < a) &=& \bbP \left(
                                               \frac{Z_{\floor{nt}}}{Z_{n+1}}
                                               < a\right)\\
\nonumber
	&=&	\bbP \left(Z_{n+1 - \floor{nt}} >  \frac{1-a}{a} Z_{\floor{nt}}\right)\\
	\label{eq:integral1}
	&=& \begin{split}\int_0^{\infty} \bbP \left(Z_{n+1 -
              \floor{nt}} > \frac{(1-a)x}{a}\right) \times \bbP \left(Z_{\floor{nt}} \in dx\right).\end{split}
	\end{eqnarray}
	Now, Chernoff's inequality implies that
	\[
	\bbP \left(Z_{n+1 - \floor{nt}} > \frac{(1-a)x}{a}\right) \leq e^{-\theta_1 \frac{(1-a)x}{a}} E\left[e^{\theta_1 Z_{n+1 - \floor{nt}}}\right].
	\]
	Since $Z_{n+1 - \floor{nt}} = \sum_{i=1}^{n+1 - \floor{nt}} \xi_i$, it follows that
	\(
	E\left[e^{\theta_1 Z_{n+1 - \floor{nt}}}\right] = (1 - \theta_1)^{-(n+1 - \floor{nt})},
	\)
	for $\theta_1 < 1$. Substituting this into \eqref{eq:integral1}, we obtain
	\[
	\bbP(T_{(\lfloor nt \rfloor)} < a) \leq (1 - \theta_1)^{-(n+1 - \floor{nt})} \int_0^\infty e^{-\theta_1\frac{1-a}{a} x} \bbP( Z_{\floor{nt}} \in dx).
	\]
	Recognize that the integral above represents the moment
        generating function of $Z_{\floor{nt}} =
        \sum_{i=1}^{\floor{nt}} \xi_i$. Since $\frac{1-a}{a} > 0$, if
        $1 > \theta_1 > \frac{a}{a - 1}$ it follows that
	\[
	\int_0^\infty e^{-\theta_1\frac{1-a}{a} x} \bbP( Z_{\floor{nt}} \in dx) = \left(1 + \theta_1 \frac{1-a}{a} \right)^{-\floor{nt}}.
	\]
	Putting things together, it follows that
	\[
	\bbP(T_{(\lfloor nt \rfloor)} < a) \leq (1-\theta_1)^{-(n+1-\floor{nt})} \left(1+\theta_1 \frac{1-a}{a}\right)^{-\floor{nt}}.
	\]
Similarly, it can be shown for any $1 \geq b > t$ that 
\[
\bbP(T^n(t) > b) \leq (1 - \theta_1)^{-(n+1 - \floor{nt})} \left(1 +
  \theta_1 \frac{1-b}{b} \right)^{- \floor{nt}},
\]
if $1 > \theta_1 > \frac{b}{b-1}$.

Thus, it follows that 
\begin{eqnarray}
\nonumber
\bbP(T^n(t) \in F) &\leq& \bbP \left( T^n(t) \in (x_-,x_+)^c \right)\\
\nonumber
  &\leq& \bbP(T^n(t) \leq x_-) + \bbP(T^n(t) \geq x_+)\\
\nonumber
  &\leq& (1 - \theta_1)^{-(n+1 - \floor{nt})} \left[\left(1 +
  \theta_1 \frac{1-x_-}{x_-} \right)^{- \floor{nt}} + \left(1 +
  \theta_1 \frac{1-x_+}{x_+} \right)^{- \floor{nt}} \right]\\
\label{eq:I_n}
  &\leq& 2 \max_{x \in F}  \left \{(1-\theta_1)^{-(n+1 - \floor{nt})} \left(1 +
  \theta_1 \frac{1-x}{x} \right)^{- \floor{nt}} \right\}.
\end{eqnarray}

Now, for any $x \in [0,1]$, it can be seen that $(1-\theta_1)^{-(n+1 - \floor{nt})} \left(1 +
  \theta_1 \frac{1-x}{x} \right)^{- \floor{nt}}$ has a unique
maximizer at $\theta_1^* = (t-x)(1-x)^{-1}$. Substituting this
into~\eqref{eq:I_n}, it follows that
	\begin{eqnarray*}
		\limsup_{n \to \infty} \frac{1}{n} \log \bbP (T_{n}(t)
                \in F) &\leq& \max_{x \in F} \left \{-(1-t)
                  \log \left( \frac{1-t}{1-x}\right) - t \log \left(\frac{t}{x}\right)\right \}\\ &=& -\inf_{x \in F} I_t(x).
	 \end{eqnarray*}
	% Optimizing over $\theta_1$, the right hand side has a
        % unique maximizer at $\theta_1^* = (1-t)(1-a)^{-1} <
        % 1$. 

	\noindent b) Next, let $G \subset [0,1]$ be an open set, such
        that $t \not \in G$ and $t < \inf\{G\}$. For each
        point $x \in G$, then there exists a $\delta > 0$ (small) such that
        $(x-\d,x+\d) \subset G$. %  Recall that $\frac{Z_{\floor{nt}}}{n}
        % \stackrel{P}{\to} t < 1$ and $\frac{Z_{n+1 - \floor{nt}}}{n} \stackrel{P}{\to} 1-t$ as $n
        % \to \infty$, by the the law of large numbers.
        Once again appealing to
        Proposition~\ref{prop:uni-expo}, we have
% Let $v > 1-t > 0$, then
% \[
% 	\bbP \left(\frac{Z_{\floor{nt}}}{n} \leq \frac{a}{1-a} \frac{Z_{n+1 - \floor{nt}}}{n}\right) \geq \int_v^{\infty} \bbP \left(\frac{Z_{\floor{nt}}}{n} \leq \frac{ax}{1-a}\right) \bbP \left(\frac{Z_{n+1 - \floor{nt}}}{n} \in dx \right),	
% \]
        \begin{eqnarray*}
          \bbP\left( T^n(t) \in (x-\d,x+\d) \right) && = \bbP\left(
            \frac{\bar Z_{\floor{nt}}}{\bar Z_{\floor{nt}} + \bar
              Z_{n+1-\floor{nt}}} \in (x-\d,x+\d) \right)\\
          && \begin{split}= \int_{z_1=0}^\infty \bbP &(\bar Z_{\floor{nt}} \in
              dz_1)\times\\
          &\bbP \bigg(\bar Z_{n+1-\floor{nt}} \in z_1 \bigg(-1 +
            \frac{1}{x+\d}, -1+\frac{1}{x-\d} \bigg) \bigg),
            \end{split}
        \end{eqnarray*}
where $\bar Z_{m(n)} := n^{-1} Z_{m(n)}$ for $m(n) \in
\{\floor{nt}, n+1-\floor{nt}\}$. Let $v > t > 0$
implying that the right hand side (R.H.S.) above satisfies
%\textcolor{red}{checl this computation, doesn't look correct to me...}
\begin{equation}
	\label{eq:det-lower-bound}
R.H.S. \geq \bbP\left(\bar Z_{\floor{nt}} \geq v \right)
\bbP\left( \bar Z_{n+1-\floor{nt}} \in v \left(-1 + \frac{1}{x+\d},
    -1+\frac{1}{x-\d} \right) \right).
\end{equation}

Now, let $\theta_1 > 0$ and consider the partial sum of `twisted' random variables $\{\xi_1^{\theta_1},\ldots,\xi_{\floor{nt}}^{\theta_1}\}$,
	\(
	Z^{\theta_1}_{\floor{nt}} = \sum_{i=1}^{\floor{nt}} \xi^{\theta_1}_i,
	\)
	where the distribution of $\xi^{\theta_1}_1$ is (by an exponential change of measure)
	\[
	\frac{\bbP(\xi^{\theta_1}_1 \in dx)}{\bbP(\xi_1 \in dx)} = \frac{e^{\theta_1 x}}{E[e^{\theta_1 \xi_1}]},
	\]
	and, by induction,
	\[
	\frac{\bbP(Z_{\floor{nt}}^{\theta_1} \in dx)}{\bbP(Z_{\floor{nt}} \in dx)} = \frac{e^{\theta_1  x}}{(E[e^{\theta_1 \xi_1}])^{\floor{nt}}}.
	\]
	Define $\Lambda_n(\theta_1) := \log(E[e^{\theta_1
          \xi_1}])^{\floor{nt}}$, and consider $\bbP(\bar Z_{\floor{nt}} > v )$. From the
        proof of Cram\'{e}r's Theorem (see \cite[Chapter 2]{DeZe2010})
        we have, for $\theta_1 > 0$, 
	\[
	 \frac{1}{n} \log \bbP (Z_{\floor{nt}} >  n v) \geq -\theta_1 v
         - \frac{\floor{nt}}{n} \log(1-\theta_1) + \frac{1}{n} \log
         \bbP \left(Z^{\theta_1}_{\floor{nt}} > n v\right).
	\]
        A straightforward calculation shows that
	\[
	\frac{1}{n} E \left[\sum_{i=1}^{\floor{nt}} \xi^{\theta_1}_i\right] = \frac{\floor{nt}}{n} \frac{1}{1-\theta_1}.
	\]
	Thus, we want to twist the random variables such that $\frac{t}{1-\theta_1} > v$, in which case 
	\(
	\bbP \left(Z^{\theta_1}_{\floor{nt}} > n v\right) \to 1
	\)
	as $n \to \infty$ by the weak law of large numbers. It follows that
	\begin{equation} 
          \label{lb-part2}
	\liminf_{n \to \infty} \frac{1}{n} \log \bbP (\bar Z_{\floor{nt}} > v ) \geq -\theta_1 v -t \log(1-\theta_1).
	\end{equation}

%%%%%%
On the other hand, consider the second probabilistic statement
in~\eqref{eq:det-lower-bound},
\[
\bbP\left(\bar Z_{n+1-\floor{nt}}\in v
          \left(-1+\frac{1}{x+\d},-1+\frac{1}{x-\d} \right)\right).
\]
Following a similar argument to that above, we consider the twisted random variables
$\{\xi_1^{\theta_2},\ldots,\xi_{n+1-\floor{nt}}^{\theta_2}\}$, and
define $\tilde \Lambda_n(\theta_2) := \log\left( \bbE[e^{\theta_2
    \xi_1}]\right)^{n+1-\floor{nt}} = -(n+1-\floor{nt}) \log \left(1 -
          \theta_2 \right)$ so that $\bbP\left(\bar Z_{n+1-\floor{nt}}\in v
          \left(-1+\frac{1}{x+\d},-1+\frac{1}{x-\d} \right)\right)$
	\begin{eqnarray}
          \nonumber
          && = \int_{v(-1+ 1/(x+\d))}^{v(-1+ 1/(x-\d))} \bbP \left(\bar
              Z_{n+1-\floor{nt}} \in dy\right)\\
          \nonumber
	&& =\int_{v(-1+ 1/(x+\d))}^{v(-1+ 1/(x-\d))} e^{-n \theta_2 y}
            \exp \left(\tilde \Lambda_n(\theta_2)\right) \bbP \left(\bar
            Z_{n+1-\floor{nt}}^{\theta_2} \in dy \right)\\
          \label{eq:lower-bound}
	&& \begin{split} \geq \exp&\left(-n\theta_2 v \left(-1+
                 \frac{1}{(x-\d)}\right) \right) ~\exp(\tilde \Lambda_n(\theta_2))\\
                 &~\bbP\left(\bar Z^{\theta_2}_{n+1-\floor{nt}} \in
                 v \left(-1+ \frac{1}{(x+\d)},-1+
                 \frac{1}{(x-\d)} \right)\right).\end{split}
	\end{eqnarray}
Observe that
	\[
	\frac{1}{n} E \left[ Z_{n+1-\floor{nt}}^{\theta_2}\right] = \frac{n+1-\floor{nt}}{n} \frac{1}{1-\theta_2}.
	\]
	Thus, we should twist the random variables such that, 
        \[
        \frac{1-t}{1-\theta_2} \in v\left(-1+ \frac{1}{(x+\d)}, -1+
                 \frac{1}{(x-\d)}\right)
        \]
implying that 
	\(
	\bbP \left(\bar Z_{n+1-\floor{nt}}^{\theta_2} \in v\left(-1+ \frac{1}{(x+\d)}, -1+
                 \frac{1}{(x-\d)}\right)\right) \to 1
	\)
	as $n \to \infty$ as a consequence of the weak law of large
        numbers. From~\eqref{eq:lower-bound} it follows that
	\begin{equation} 
          \label{lb-part1}
          \begin{split}
		\liminf_{n \to \infty} \frac{1}{n} \log \bbP \bigg(
                  \bar Z_{n+1-\floor{nt}} &\in v\left(-1+ \frac{1}{(x+\d)}, -1+
                 \frac{1}{(x-\d)}\right) \bigg)\\ &\geq -\theta_2 v
             \left(-1 + \frac{1}{x+\d} \right) - (1-t)
             \log(1-\theta_2).
            \end{split}
	\end{equation}
%%%%%%%%%
	Using the limits in \eqref{lb-part1} and \eqref{lb-part2} it
        follows that for any $0 < \eps < \d$,
	\[
        \begin{split}
	 \liminf_{n\to \infty} \frac{1}{n} \log \bbP(T^n(t) & \in
         (x-\eps,x+\eps))\\ &\geq -\theta_1 v -t \log(1-\theta_1) -\theta_2 v
         \left(-1 + \frac{1}{x+ \eps} \right)- (1-t) \log(1-\theta_2).
         \end{split}
	\]
	This is valid for any $v > t$. In particular, setting $v = x-\eps$ we obtain 
	\[
        \begin{split}
	 \liminf_{n\to \infty} \frac{1}{n} \log \bbP(T^n(t) \in
         (x-\eps,x+\eps)) \geq &-\theta_1 (x-\eps) - t \log(1-\theta_1)\\ &-
         \theta_2 \left( \frac{x-\e}{x+\eps} \right) (1-(x+\eps)) - (1-t)
         \log(1-\theta_2).
         \end{split}
	\]

	Now, consider the function $I(\theta_1,\theta_2) := \theta_1 (x-\eps) + t \log(1-\theta_1) +
         \theta_2 \left( \frac{x-\eps}{x+\eps} \right) (1-(x+\eps)) + (1-t)
         \log(1-\theta_2)$. For $ \theta_2,\theta_1 < 1 $, it is
         straightforward to see that the Hessian is positive semi
         definite, implying it is convex. The unique minimizer of $I(\theta_1,\theta_2)$ is $(\theta_1^*,\theta_2^*) = \left(1 - t/(x-\eps), 1 - \left(\frac{x+\eps}{x-\eps}\right)\left(\frac{1-t}{1-(x+\eps)}\right) \right)$. Letting $\e \to 0$,
         it follows that 
	\[
	 I(\theta_1^*,\theta_2^*) = t \log \left(\frac{t}{x} \right) +
         (1-t) \log \left(\frac{1-t}{1-x}\right) = I_t(x).
	\]
        Thus, we have
        \begin{equation}
          \label{eq:local-LD-lower}
          \liminf_{n\to\infty} \frac{1}{n} \log \bbP(T^n(t) \in
          (x-\d,x+\d)) \geq -I_t(x).
        \end{equation}

        Next, it follows by definition that, for small enough $\d >
        0$,
        \[
        \frac{1}{n} \log \bbP(T^n(t) \in G) \geq
        \sup_{x \in G} \frac{1}{n} \log \bbP(T^n(t) \in (x-\d,x+\d)),
        \]
        implying that
        \[
        \liminf_{n\to\infty} \frac{1}{n} \log \bbP(T^n(t) \in G) \geq
        -\inf_{x\in G} I_t(x).
        \]
        On the other hand, if $t > \sup\{G\}$, we will now consider a $v >
        1-t > 0$ in the lower bounding argument used
        in~\eqref{eq:det-lower-bound}. Since the remaining arguments
        are identical to the previous steps, we will not repeat them. This proves the LD lower bound.
        
        Finally, observe that the rate function $I_t$ is continuous
        and convex. Consider the level set $L(c) = \{x \in [0,1] : I_t(x)
        \leq c\}$, for $c > 0$. Let $\{x_n,~n\geq1\}$ be a sequence
        points in the set $L(c)$ such that $x_n \to x^* \in (0,1)$ as $n \to
        \infty$. Since $I_t$ is continuous, it follows that
        $I_t(x_n) \to I_t(x^*)$ as $n \to \infty$. Suppose $I_t(x^*) >
        c$, then the only way this can happen is if there is a
        singularity at $x^*$. However, this contradicts the fact that
        $I_t$ is continuous on the domain $(0,1)$, implying that $x^*
        \in L(c)$. Therefore, it is
        the case that $L(c)$ is closed. Furthermore, this level set is
        bounded (by definition), implying that it is compact. Thus,
        $I_t$ is a good rate function as well.
$\QED$
\end{proof}

Now, suppose that $\{\tilde T_{(i)},~i \leq n\}$ are the ordered statistics
of random variables with distribution $F$
(assumed to have positive support) that is absolutely continuous with
respect to the Lebesgue measure, and strictly increasing. Define 
\begin{equation} \label{eq:rate-function-OS-general}
\tilde I_t(y) := \inf_{x \in [0,1]: F^{-1}(x) = y} I_t(x).
\end{equation}
The following corollary establishes an LDP for
the corresponding order statistics process.

\begin{corollary}\label{cor:cramer-order}
  Fix $t \in [0,1]$. Then, the ordered statistics process
  corresponding to $\{\tilde T_{(i)},~i \leq n\}$ satisfies the LDP with good
  rate function $\tilde I_t$.
\end{corollary}
Since $F^{-1}$ maps $[0,1]$ to $[0,\infty)$, which are Hausdorff
spaces, the proof is a simple application of the contraction principle
\cite[(2.12)]{LySe1987}. For the remainder of the
paper, however, we will operate under the assumption that the arrival
epochs are i.i.d. uniform random variables. The analysis below can be straightforwardly extended to the more
general case where the distribution is absolutely continuous with
respect to the Lebesgue measure.

\subsection{LDP for the offered load}
Next, recall that $\{\nu_i, i \geq 1\}$ is a sequence of i.i.d. random
variables with cumulant generating function $\varphi(\theta) =
\log \bbE[e^{\theta \nu_1}] < \infty$ for some $\theta \in \bbR$.% ; let $\bbD
% := \{\theta \in \bbR : \varphi(\theta) < \infty\}$. We will also
% assume that $0 \in \bbD$, the interior of $\bbD$.
~The next theorem
shows that the service process $(S^n(t),~t \in [0,1])$ satisfies the
LDP.

\begin{lemma} [Cram\'{e}r's Theorem \cite{DeZe2010}]
\label{thm:cramer-service}
 Fix $t \in [0,1]$. Then, the sequence of random variables
 $\{S^n(t),~n\geq 1\}$ satisfies the LDP with good rate function
 $\Lambda_t^*(x) := \sup_{\theta \in \bbR} \{\lambda x - t \varphi(\theta)\}$.
  % \noindent a) For any closed set $F \subset [0,\infty)$,
  % \[
  % \limsup_{n\to \infty} \frac{1}{n} \log \bbP(S^n(t) \in F) \leq -
  % \inf_{x \in F} \Lambda_t^*(x), ~\text{and}
  % \]

  % \noindent b) For any open set $G \subset [0,\infty)$,
  % \[
  % \liminf_{n\to \infty} \frac{1}{n} \log \bbP(S^n(t) \in G) \leq
  % -\inf_{x \in G} \Lambda_t^*(x).
  % \]
\end{lemma}
Note that we specifically assume that $0 \in \bbD$, since \cite[Lemma 2.2.5]{DeZe2010} shows that if $\bbD =\{0\}$, then $\Lambda_t^*(x)$ equals zero for all $x$. \cite[Lemma
2.2.20]{DeZe2010} proves that the rate function is good if the interior
condition is satisfied. We now establish an LDP for the offered load
process $(X^n(t) = S^n(t) - T^n(t),~t \geq 0)$ by leveraging
Theorem~\ref{thm:cramer-order} and
Lemma~\ref{thm:cramer-service}. 

\begin{proposition} \label{thm:marked-order}
  Fix $t \in [0,1]$, and let $\sX := [0,1] \times [0,\infty)$. Then, the sequence of random variables
  $\{X^n(t),~n\geq 1\}$ satisfies the LDP with good rate function
  $J_t(y) = \inf_{\{x \in \sX: x_1 =  x_2 + y\}} I_t(x_1) +
  \Lambda_t^*(x_2)$ for $y \in \bbR$.
% \noindent a) For any closed set $F \subset [0,\infty)$,
% \[
% \limsup_{n\to \infty} \frac{1}{n} \log \bbP(X^n(t) \in F) \leq -
% \inf_{y \in F} J_t(y),~\text{and}
% \]

% \noindent b) For any open set $G \subset [0,\infty)$,
% \[
% \liminf_{n \to \infty} \frac{1}{n} \log \bbP(X^n(t) \in G) \geq
% -\inf_{y \in G} J_t(y).
% \]
\end{proposition}

\begin{proof}
  % Recall that $S^n(t)$ and $T^n(t)$ are independent and satisfy LDP's
  % as demonstrated in Theorem~\ref{thm:cramer-service} and
  % Theorem~\ref{thm:cramer-order}.
  \cite[Lemma 2.6]{LySe1987} implies
  that $\{S^n(t),~n\geq 1\}$ and $\{T^n(t),~n\geq 1\}$ are LD tight
  (as defined in Definition~\ref{def:ld-tight}). By \cite[Corollary  2.9]{LySe1987} it follows that ${(S^n(t),T^n(t)),~n\geq 1}$ satisfy
  the LDP with good rate function $\tilde I_t(x_1,x_2) = I_t(x_1) +
  \Lambda_t^*(x_2)$. Now, since subtraction is trivially continuous on
  the topology of pointwise convergence, it follows that $\{X^n(t) =
  S^n(t) - T^n(t), ~n\geq 1\}$ satisfies the LDP with rate function $J_t$ as a
  consequence of the contraction principle (see
  \cite[(2.12)]{LySe1987}).
$\QED$
\end{proof}

As an example of the rate function, suppose the service times are
exponentially distributed with mean 1. Then, we have
\[
J_t(y) = \inf_{x \in [0,1]} \left\{ t \log\left(\frac{t}{x} \right) +
  (1-t) \log\left(\frac{1-t}{1-x}\right) + t \log \left(
    \frac{t}{x-y}\right) + (x-y-t) \right \}.
\]
 Some (tedious) algebra shows that $J_t(y)$ is strictly convex, and
 thus has a unique minimizer, which is the solution to the cubic
 equation
\[
x^3 - y x^2  - 2 t x +ty = 0.
\]
Unfortunately, the sole real solution to this cubic equation has a
complicated form, which we do not present, but can be found by using a
symbolic solver.

\section{An LDP for the Workload} \label{sec:workload}
% Next, we prove that the linearly interpolated process $\tilde W^n$ is exponentially equivalent to
% $W^n$. 
Recall that $W^n(t) = \Gamma(X^n)(t) = \sup_{0 \leq s \leq
  t}(X^n(t) - X^n(s))$. The key difficulty in establishing the LDP for $W^n(t)$ is the fact
that while $\Gamma$ is continuous on the space $(\sD[0,1], \sU)$,
the latter is not a Polish space. Therefore, it is not possible to directly apply the
contraction principle to $\Gamma$ to establish the LDP. Consider, instead, the continuous process $(\tilde W^n(t),~t\in
[0,1])$, formed by linearly interpolating between the jump levels of
$W^n$; equivalently, $\tilde W^n = \Gamma(\tilde X^n)$, where $\tilde X^n$
is the linearly interpolated version of the offered load. We first show that $(\tilde W^n(t),~t \in [0,1])$ is
asymptotically exponentially equivalent to $(W^n(t),~t\in
[0,1])$. Next, we prove that, for each fixed $t \in [0,1]$, $(\tilde
X^n(s),~s \in [0,t])$ satisfies the LDP, via a projective limit. This
enables us to prove that $\tilde W^n(t)$
satisfies the LDP by invoking the contraction principle with $\Gamma$
and using the fact that $(\sC[0,t],\sU)$ is a Polish space. Finally, by \cite[Theorem 4.2.13]{DeZe2010} the
exponential equivalence of the processes implies that $W^n(t)$
satisfies the LDP with the same rate function. 

\subsection{Exponential Equivalence}
We define the linearly interpolated service time process as
\[
\tilde S^n(t) := S^n(t) + \left(t - \frac{\floor{nt}}{n} \right) \nu_{\floor{nt}+1},
\]
and the linearly interpolated arrival epoch process as
\[
\tilde T^n(t) := T^n(t) + \left(t - \frac{\floor{nt}}{n} \right)
(T_{(\floor{nt}+1)} - T_{(\floor{nt})}).
\]
Define $\Delta_{n,t} := T_{(\floor{nt}+1)} - T_{(\floor{nt})}$ and
note that these are spacings of ordered statistics. The process $\tilde X^n = \tilde S^n -
\tilde T^n \in \sC[0,1]$ can now be used to define the interpolated
workload process $\tilde W^n = \Gamma(\tilde X^n) \in
\sC[0,1]$. Recall that $\|\cdot\|$ is the supremum norm on $\sC[0,1]$.

\begin{proposition} \label{prop:exponential-eq}
  The processes $\tilde W^n$ and $W^n$ are exponentially
  equivalent. That is, for any $\delta > 0$
  \[
  \limsup_{n\to \infty} \frac{1}{n} \log \bbP(\| \tilde W^n - W^n \| >
  \delta) = -\infty.
  \]
\end{proposition}

\begin{proof}
  First, observe that for each $t \in [0,1)$
\begin{eqnarray*}
\left| S^n(t) - \tilde S^n(t) \right| &\leq& \left| \left(t -
    \frac{\floor{nt}}{n} \right) \nu_{\floor{nt}+1}\right|\\
  &\leq& \frac{\nu_{\floor{nt}+1}}{n},
\end{eqnarray*}
and $S^n(1) = \tilde S^n(1)$ by definition. Similarly, 
\[
\left| T^n(t) - \tilde T^n(t) \right| \leq \left(t -
  \frac{\floor{nt}}{n} \right) \left(T_{(\floor{nt}+1)}
  - T_{(\floor{nt})} \right) = \frac{nt - \floor{nt}}{n} \D_{n,t}
~\forall t \in [0,1),
\]
and $T^n(1) = \tilde T^n(1)$. Now, let $\{E_1,\ldots,E_{n+1}\}$ be a collection of independent unit
mean exponential random variables, and define $Z_{n+1} :=
\sum_{i=1}^{n+1} E_i$. Recall (from \cite[pp. 134-136]{DaNa2003}, for instance) that the spacings of the uniform ordered statistics are
equal in distribution to the ratio 
\(
\D_{n,t} \stackrel{D}{=} \frac{E_1}{Z_{n+1}}.
\)
It follows that
\begin{eqnarray*}
  \left\| X^n - \tilde X^n \right\| &\leq& \left\| S^n - \tilde
                                          S^n\right\| + \left\| T^n -
                                          \tilde T^n \right\|\\
&\leq& \left \| \frac{\nu_{\floor{nt}+1}}{n} \right\| + \left\| \left(\frac{nt -
       \floor{nt}}{n} \right) \D_{n,t}\right\|.
\end{eqnarray*}

Now, consider the measure of the event $\left \{\left\| X^n - \tilde X^n
\right\| > 2 \delta \right\}$, and use the inequality above to obtain:
\begin{eqnarray}
\nonumber
\bbP\left( \left\| X^n - \tilde X^n
\right\| > 2 \delta\right) &\leq& \bbP\left( \left \| \frac{\nu_{\floor{nt}+1}}{n} \right\| + \left\| \left(\frac{nt -
       \floor{nt}}{n} \right) \D_{n,t}\right\|>2 \d
                                  \right)\\
  \nonumber 
&\leq& \bbP\left( \left \| \frac{\nu_{\floor{nt}+1}}{n} \right\| > \d
       \right) + \bbP\left( \left\| \left(\frac{nt -
       \floor{nt}}{n} \right) \D_{n,t}\right\|>  \d
                                  \right)\\
\label{eq:offered-exp-eq}
&\leq& n \bbP \left(\nu_1 > n \d \right) + n \bbP \left(\D_{n,1} >
       n\d \right),
\end{eqnarray}
where $\bbP\left(\left\| \nu_{\floor{nt}+1} \right\| > n \d\right) =
\bbP(\left( \sup_{0 \leq t < 1} \nu_{\floor{nt}+1} > n \d \right) =
\bbP(\cup_{m=1}^n \{\nu_i > n \d\})
\leq n \bbP(\nu_1 > n \d)$ follows from the union bound and the fact that the
service times are assumed i.i.d. Similarly, since $\D_{n,t}
\stackrel{D}{=} E_1/Z_{n+1}$ and
$n^{-1} > n^{-1}(nt - \floor{nt})$ for all $t \in [0,1]$ and $n \geq 1$, we obtain the bound on
$\D_{n,t}$ by similar arguments. Note that we have abused notation
slightly in~\eqref{eq:offered-exp-eq} and re-used $\D_{n,m}
= T_{(m)} - T_{(m-1)}$ for $m \in \{1,\ldots,n\}$ with the
understanding that $T_{(0)}=0$.

Using Chernoff's inequality to obtain
\begin{equation*}
\bbP \left( \nu_1 > n \d\right) \leq e^{-n \d \theta_1}
e^{\varphi(\theta_1)},\\
\end{equation*}
so that
\begin{equation} \label{eq:service-exp-eq}
\limsup_{n\to\infty} \frac{1}{n} \log \left(n \bbP( \nu_1 > n \d) \right) \leq
-\theta_1 \d ,
\end{equation}
for all $\theta \in \bbR$. On the other hand, we have
\begin{eqnarray*}
\bbP \left(\D_{n,1} > n \d \right) &=&
                                                               \bbP\left(
                                                               E_1
                                                               \left(1
                                                               -
                                                               n\d
                                                               \right)
                                                               >
                                                               n\d Z_n\right)\\
&=& \int_0^\infty \bbP\left( E_1 > x \frac{n \d}{1-n \d} \right)
    \bbP(Z_{n} \in dx),
\end{eqnarray*}
using the fact that
$\{E_i,~1\leq i \leq n\}$ are i.i.d. Again, using Chernoff's inequality, the right hand side (R.H.S.)
satisfies
\begin{eqnarray*}
  R.H.S. &\leq& \frac{1}{1-\theta_2} \int_0^\infty \exp\left(
                -\theta_2 x \frac{n\d}{1-n\d} \right) \bbP(Z_{n} \in dx) ~\forall \theta_2
  < 1\\
  &=& \frac{1}{1-\theta_2} \left(\frac{1 - n \d}{1 - n\d(1-\theta_2)} \right)^n.
\end{eqnarray*}
It follows that 
\(
\limsup_{n\to \infty} \frac{1}{n} \log \bbP\left( \D_{n,1} > n\d \right) 
\)
\begin{eqnarray}
\nonumber
&\leq &\limsup_{n \to \infty}
                                         \frac{\log(1-\theta_2)}{n}+
                                         \limsup_{n\to\infty} \log
        \left(\frac{1- n \d}{1 - n\d(1-\theta_2)} \right)\\
\nonumber
&=& -\log(1-\theta_2),
\end{eqnarray}
and so
\begin{equation}
  \label{eq:arrival-exp-eq}
  \limsup_{n\to\infty} \frac{1}{n} \log \left(n \bbP\left( \D_{n,1} > n\d \right) \right) \leq -\log(1-\theta_2)
~\forall~\theta_2 < 1.
\end{equation}

Now,~\eqref{eq:offered-exp-eq},~\eqref{eq:service-exp-eq} and ~\eqref{eq:arrival-exp-eq},
together with the principle of the largest term (\cite[Lemma
1.2.15]{DeZe2010}), imply 
\[
\limsup_{n\to\infty} \frac{1}{n} \log \bbP\left( \left\| X^n - \tilde
    X^n \right\| \right) \leq \max \{ -\theta_1 \d, -\log(1-\theta_2)\}.
\]
Since $\theta_1 \in \bbR$ and $\theta_2 \in (-\infty,1)$, by letting $\theta_1 \to \infty$ and $\theta_2 \to
-\infty$ simultaneously, it follows that
\begin{equation} \label{eq:offered-exp-eq-2}
\limsup_{n\to\infty} \frac{1}{n} \log \bbP\left( \left\| X^n - \tilde
    X^n \right\| > 2\d\right) = -\infty.
\end{equation}

Finally, using the fact that the map $\Gamma$ is Lipschitz in
$(\sD,J_1)$ (see \cite[Theorem 13.5.1]{Wh01}) we have
\begin{equation*}
  \bbP\left(\left\| W^n - \tilde W^n \right\| > 4 \d\right) \leq \bbP\left( \left\| X^n - \tilde
    X^n \right\| > 2\d\right),
\end{equation*}
and thus
\begin{equation}
  \label{eq:workload-exp-eq}
  \limsup_{n\to\infty} \frac{1}{n} \log \bbP \left(\left\| W^n - \tilde W^n \right\| > 4 \d \right) = -\infty.
\end{equation}
Since $\d > 0$ is arbitrary, the theorem is proved.
$\QED$
\end{proof}

\subsection{Sample path LDP for the offered load}
First, we prove
an LDP for the increments of the offered load process. Fix $t \in [0,1]$, and consider an arbitrary $d$-partition of $[0,t]$,
$j := \{ 0 \leq t_1 < t_2 < \cdots < t_d \leq t\}$, so that the
increments are $\mathbf \D_n^X(j) = \mathbf \D_n^S(j) - \mathbf
\D_n^T(j)$, where
\begin{equation}
  \label{eq:increments-OS}
  \mathbf \D^T_n(j) := (T^n(t_1),T^n(t_2)- T^n(t_1),\ldots, T^{n}(t) -
T^n(t_d)),
\end{equation}
and 
\[
\mathbf \D_n^S(j) = \left(S^n(t_1), S^n(t_2)-S^n(t_1),\ldots, S^n(t)-S^n(t_d) \right).
\]

Now, using 
representation~\eqref{eq:order-expo}, it follows that 
\[
\mathbf \D_n^T(j) \stackrel{D}{=} \frac{1}{Z_{n+1}}\left( Z_{\floor{nt_1}},
  Z_{\floor{n t_2}} - Z_{\floor{n t_1}}, \ldots, Z_{\floor{n t}} - Z_{\floor{n t_d}}  \right)
\]
A straightforward calculation shows that the cumulant generating function of the $(d+1)$-dimensional random
vector $\mathbb{Z}_n := (Z_{\floor{nt_1}}, \ldots, Z_{\floor{n t}} - Z_{\floor{n t_d}}, Z_{n+1})$ satisfies
\begin{equation} \label{eq:cumulant-limit}
\lim_{n\to\infty}\frac{1}{n}\log \bbE \left[\exp\left( \langle\mathbf
    \mathbf \lambda,
    \mathbb{Z}_n\rangle\right)\right] = \Lambda(\l) \text{ for }
\mathbf \l
\in \bbR^{d+1},
\end{equation}
where
\[
\Lambda_j(\mathbf \l) := \begin{cases}
 - \sum_{i=1}^d (t_i - t_{i-1})
\log(1-\l_i-\l_{d+1}) - (1-t) \log(1-\l_{d+1}), &~ \mathbf \l \in \bbD_{\Lambda},\\
 +\infty, & \mathbf \l \not \in \bbD_{\Lambda},
\end{cases}
\]
and $\bbD_{\Lambda} := \{\mathbf \lambda \in \bbR^{d+1} : \max_{1\leq i \leq
  d}\lambda_i + \lambda_{d+1} < 1, ~\text{and}~ \lambda_{d+1} < 1
\}$; note, $t_0 := 0$. We also define the function
\[
\begin{split}
\Lambda_j^*(\mathbf{x}) := \sup_{\mathbf \l \in \bbD_{\Lambda}} \sum_{i=1}^d
\left( \l_i + \l_{d+1} \right) x_i &+ (t_i - t_{i-1}) \log(1-\l_i -
\l_{d+1}) \\&+ \l_{d+1} x_{d+1} + (1-t) \log(1-\l_{d+1}).
\end{split}
\]
Now, define the continuous function $\Phi : \bbR^{d+1} \to \bbR^d$ as $\Phi(\mathbf x) =
(x_1,\ldots,x_{d})/{\sum_{i=1}^{d+1} x_i}$. We can now state the LDP
for the increments $\mathbf \D^T_n(j)$.

\begin{lemma}\label{lem:increments-os}
  Let $j := \{0 \leq t_1 < t_1 < \cdots < t_{d} \leq t\}$ be an
  arbitrary partition of $[0,t]$. Then the increments of the ordered
  statistics process, $\mathbf \D^T_n(j)$, satisfy the LDP with good rate
  function $\hat \Lambda_j(\mathbf y) = \inf_{\mathbf x \in \bbR^{d+1} :
    \Phi(\mathbf x) = \mathbf y} \Lambda_j^*(\mathbf x)$ for all
  $\mathbf y \in (0,1]^d$. Furthermore, 
  \[
  \hat \Lambda_j(\mathbf y) = \sum_{i=1}^d (t_i - t_{i-1}) \log \left(
    \frac{t_i - t_{i-1}}{y_i}\right) + (1-t) \log \left(
    \frac{1-t}{1-\sum_{l=1}^d y_l} \right).
  \]
\end{lemma}

\begin{proof}
Equation~\eqref{eq:cumulant-limit} implies that the sufficient
conditions of the Gartner-Ellis Theorem \cite[Theorem 2.3.6]{DeZe2010}
are satisfied, so that $\mathbf{\mathbb Z}_n$ satisfies the LDP with
rate function $\Lambda_j^*$. Equivalently, the random vector $(Z_{\floor{nt_1}},
  Z_{\floor{n t_1}} - Z_{\floor{n t_2}}, \ldots, Z_{\floor{n t}} -
  Z_{\floor{n t_d}}, Z_{n+1} - Z_{\floor{nt}})$ satisfies the LDP with
  good rate function $\Lambda_j^*$. Now, since $\bbR^{d+1}$ and $\bbR^d$
  are Polish spaces, the contraction principle applied
  to the map $\Phi$ yields the LDP. Finally, it is straightforward to
  check that the Hessian of $\Lambda_j^*(\mathbf x)$ is positive
  semi-definite, implying that the latter is convex. It can now be
  seen that the minimizer $\mathbf x^*$ is such that
  $\sum_{j=1}^{d+1} x^*_j = 1$ and $x_i^* = y_i$, for a
  given $\mathbf y \in (0,1]^d$. The final expression for $\hat \Lambda_j(\mathbf
  y)$ follows.
$\QED$
\end{proof}

As a sanity check, we show that if $d=1$ the rate function $\hat
\Lambda_j(\mathbf y)$ is
precisely the rate function $I_t$ in~\eqref{eq:rate-function-OS}.

\begin{corollary}
  Let $j = \{0 \leq t_1 \leq t\}$ and $d=1$, then the rate function is
  \[
  \hat \Lambda_j(y) = t_1 \log \left( \frac{t_1}{y} \right) + (1-t_1)
  \log \left( \frac{1-t_1}{1-y} \right), ~\forall y \in (0,1).
  \]
\end{corollary}
\begin{proof}
Since $d = 1$, by definition we have for all $\mathbf x \in \bbR^2$
\[
\Lambda_j^*(\mathbf x) = (\lambda_1 + \lambda_2) x_1 + t_1 \log (1 -
\lambda_1 - \lambda_2) + \lambda_2 x_2 + (1-t_1) \log(1-\lambda_2).
\]
Substituting the unique maximizer $(\lambda^*_1, \lambda^*_2) =
\left(\frac{1-t_1}{x_2} - \frac{t_1}{x_1}, 1 - \frac{1-t_1}{x_2}
\right)$, it follows that
\begin{eqnarray*}
\Lambda_j^*(\mathbf x) &=& (x_1 + x_2 - 1) + t_1 \log \left(
  \frac{t_1}{x_1}\right) + (1-t_1) \log \left(
                         \frac{1-t_1}{x_2}\right).
\end{eqnarray*}
% where we have used the definition of $y = \Phi(\mathbf x) = x_1/(x_1 +
% x_2)$ in the last equality.
Finally, using the
fact that $\mathbf x^* = \arg \inf \left\{\Lambda^*(\mathbf x) \right\}$ satisfies $x_1^* + x_2^* = 1$, the
corollary is proved.
$\QED$
\end{proof}

As an aside, note that this result shows that
Theorem~\ref{thm:cramer-order} could also be established as a corollary
of Lemma~\ref{lem:increments-os}. However, while the proof is straightforward, it is also somewhat `opaque': the proof of
Theorem~\ref{thm:cramer-order} explicitly demonstrates how the
long-range dependence inherent in the order statistics process affects the LDP and is, we believe, more
clarifying as a consequence. Next, we use this result to prove a sample
path LDP for the ordered statistics process $(\tilde T^n(s),~s \in [0,t])$ (for
each fixed $t$) in the topology of pointwise convergence on the space
$\sC[0,t]$. Observe that the exponential equivalence of $\tilde T^n$
and $T^n$ implies that the increments of $\tilde T^n$ satisfy the LDP
in Lemma~\ref{lem:increments-os}.

% Note that~\eqref{eq:offered-exp-eq-2} implies that $X^n(t)$ and
% $\tilde X^n(t)$ (for each fixed $t\in[0,1]$) are exponentially
% equivalent. Thus, $\tilde X^n(t)$ satisfies the LDP with rate function
% $J_t$ as defined in Proposition~\ref{thm:marked-order} and the
% increments of $\tilde X^n$ satisfy the LDP in
% Lemma~\ref{lem:increments-offered}, which will now be leveraged to establish
% an LDP for the sample path $(\tilde X^n(s), ~s\in [0,t])$, for every fixed $t
% \in [0,1]$. 
Let $\sJ_t$ be the space of
all possible finite partitions of $[0,t]$. Note that for each partition
$j = \{0 \leq t_1 < t_1 < \cdots < t_{d} \leq t\} \in \sJ_t$, the
increments take values in the space $[0,1]^d$ which is Hausdorff. Thus,
we can appeal to the Dawson-Gartner
theorem~\cite[Theorem 4.6.1]{DeZe2010} to establish the LDP for the
sample path $(\tilde T^n(s), ~s\in [0,t])$ via a projective limit. Let $p_j: \sC[0,t] \to \bbR^{|j|}$ be the canonical projections of the
coordinate maps, and $\sX$ be space of all functions in $\sC[0,t]$
equipped with the topology of pointwise convergence. Recall that any non-decreasing continuous function $\phi \in
\bar \sC[0,t]$ is of bounded variation, so that $\phi = \phi^{(a)} +
\phi^{(s)}$ by the Lebesgue decomposition theorem; here, $\phi^{(a)}$ is the absolutely continuous component
and the $\phi^{(s)}$ is the singular component of $\phi$. Recall, too, that a
singular component has derivative that satisfies $\dot \phi^{(s)}(t) =
0$ a.e. $t$. 

\begin{lemma}
  Fix $t \in [0,1]$. Then the sequence of sample paths $\{(\tilde
  T^n(s),~s \in [0,t]), ~n \geq 1\}$ satisfies the LDP with good rate
  function
  \[
  \hat \L_t(\phi) = 
    -\int_0^t \log\left( \dot \phi^{(a)}(s)\right) ds + (1-t) \log \left(
      \frac{1-t}{1-\phi(t)} \right) ~\forall \phi \in \bar \sC[0,t].
  \]
\end{lemma}

\begin{proof}
The proof largely follows that of \cite[Lemma 5.1.8]{DeZe2010}. There are two steps to establishing this result. First, we must show
that the space $\sX$ coincides with the projective limit $\tilde \sX$
of $\{\sY_j = \bbR^{|j|},~j\in \sJ_t\}$. This, however, follows immediately from the proof of
\cite[Lemma 5.1.8]{DeZe2010}. Second, we must argue that
\[
\tilde \L_t(\phi) := \sup_{0 \leq t_1 < \ldots < t_k \leq t} \sum_{l=1}^k
(t_l - t_{l-1}) \log \left( \frac{t_l - t_{l-1}}{\phi(t_l) -
    \phi(t_{l-1})} \right) + (1-t) \log\left( \frac{1-t}{1-\phi(t)} \right)
\]
is equal to $\hat \L_t(\phi)$. Without loss of generality, assume that
$t_k = t$. Recall that $\phi$ has bounded
variation, implying that $\phi^{(a)}(t) = \int_0^t \dot \phi(s) ds$
or, equivalently, $\dot\phi^{(a)}(s) = \dot \phi(s)$ a.e. $s \in [0,t]$. Since $\log(\cdot)$ is concave, 
Jensen's inequality implies that 
\begin{eqnarray*}
\sum_{l=1}^k (t_l - t_{l-1}) \log \left( \frac{\phi(t_l) -
    \phi(t_{l-1})}{t_l - t_{l-1}} \right) &=& \sum_{l=1}^k(t_l -
                                              t_{l-1}) \log
                                              \left(\frac{\int_{t_l}^{t_{l-1}}
                                                \dot \phi(r) dr}{t_l
                                              - t_{l-1}} \right)\\
  &\geq& \sum_{l=1}^k \int_{t_l}^{t_{l-1}} \log\left(\dot \phi(r)
         \right) dr\\
  &=& \int_0^t \log \left(\dot \phi^{(a)}(r) \right) dr,
\end{eqnarray*}
so that $\tilde \L_t(\phi) \leq \hat \L_t(\phi)$. 

Next, define
\[
\phi_n(r) = n \left( \phi^{(a)} \left( \frac{[nr]+1}{n}\right) - \phi^{(a)}
  \left( \frac{[nr]}{n}\right)\right) + n \left( \phi^{(s)} \left( \frac{[nr]+1}{n}\right) - \phi^{(s)}
  \left( \frac{[nr]}{n}\right)\right),
\]
and observe that
\[
\lim_{n \to\infty} \phi_n(r) = \dot \phi^{(a)}(r) ~\text{a.e.}~r\in [0,t],
\]
since $n \left( \phi^{(a)} \left( \frac{[nr]+1}{n}\right) - \phi^{(a)}
  \left( \frac{[nr]}{n}\right)\right) \to \dot \phi^{(a)}(r)$ and $n \left( \phi^{(s)} \left( \frac{[nr]+1}{n}\right) - \phi^{(s)}
  \left( \frac{[nr]}{n}\right)\right) \to \dot \phi^{(s)}(r) = 0$
a.e. $r \in [0,t]$ as $n \to \infty$. Now, consider the uniform
partition $0 = t_0 < t_1 < \ldots < t_n = t$ of $[0,t]$, where $t_l =
t l/n$, so that $\liminf_{n \to \infty}\sum_{l=1}^n \frac{1}{n}\log \left( n (\phi(t_l) -
    \phi(t_{l-1}))\right)$
\begin{eqnarray*}
&&=\liminf_{n \to \infty}\sum_{l=1}^n \frac{1}{n}\log \bigg( n (\phi^{(a)}(t_l) -
    \phi^{(a)}(t_{l-1})) +  n (\phi^{(s)}(t_l) -
    \phi^{(s)}(t_{l-1}))\bigg)\\
  && = \liminf_{n\to\infty} \int_0^t \log\left(\phi_n(r) \right) dr\\
  && \geq \int_0^t \liminf_{n\to\infty}\log(\phi_n(r)) dr\\
  && = \int_0^t \log\left(\dot \phi^{(a)}(r) \right) dr,
\end{eqnarray*}
where the inequality follows from Fatou's Lemma and the last equality
is a consequence of the continuity of $\log(\cdot)$. Now, by definition, 
\[
\tilde \L_t(\phi) \geq \liminf_{n \to \infty} -\sum_{l=1}^n \frac{1}{n}\log \left( n (\phi(t_l) -
    \phi(t_{l-1}))\right) + (1-t) \log \left( \frac{1-t}{1-\phi(t)}\right),
\]
implying that
\[
\tilde \L_t(\phi) \geq -\int_0^t \log\left(\dot \phi^{(a)}(r) \right)
dr + (1-t) \log \left( \frac{1-t}{1-\phi(t)}\right) = \hat \L_t(\phi).
\]
$\QED$
\end{proof}

% Now, consider the function
% \begin{equation*}
% \hat {\mathbf I}_j(\mathbf x) := \sum_{l=1}^d (t_l - t_{l-1}) \Lambda^* \left(\frac{x_l - x_{l-1}}{t_l - t_{l-1}} \right), 
% \end{equation*}
% where
% $\Lambda^*(z) := \sup_{\l \in \bbR} \{ \l z - \varphi(\lambda) \}$ is
% the Fenchel-Legendre transform of the cumulant generating function of
% the service times. 
For the service process, we consider the following result implied by
\cite{Pu1995}. As noted in
\cite{DuMaTo2010}, the form of Mogulskii's theorem presented in
\cite[Theorem 5.1.2]{DeZe2010} does not cover the case of
exponentially distributed service times, thus we appeal to the
generalization proved in \cite{Pu1995}. Note that \cite{Pu1995} proves the result in the
$M_1$ topology on the space $\sD[0,t]$ which implies convergence
pointwise as required here.

\begin{lemma}
  Fix $t \in [0,1]$. Then the sequence of sample paths $\{(\tilde
  S^n(s),~s\in[0,t])\}$ satisfies the LDP with good rate function, for
  each $\psi \in \bar \sC[0,t]$,
  \[
  \hat I_t(\psi) = \int_0^t \L^*(\dot \psi^{(a)}(s)) ds + \psi^{(s)}(t).
  \]
\end{lemma}

These two results now imply the LDP for the sequence of sample paths
$\{(\tilde X^n(s),~s\in[0,t])\}$. % Let $\bar \sC^0[0,t] := \bar
% \sC[0,t] \cap \sA[0,t]$ be the space of all absolutely continuous
% non-decreasing functions with domain $[0,t]$.
\begin{proposition} \label{prop:sample-path-offered}
   % Let $j := \{0 \leq t_1 < t_1 < \cdots < t_{d} \leq t\}$ be an
  % arbitrary partition of $[0,t]$. Then, the increments of the offered
  % load process, $\mathbf \D^X_n(j)$, satisfy the LDP with good rate
  % function $\hat J_j(\mathbf y) := \inf_{\{\mathbf x = (\mathbf x_1,
  %   \mathbf x_2) \in [0,1]^d \times \bbR^d : \mathbf x_1 - \mathbf x_2
  %   = \mathbf y\}} \hat \Lambda_j(\mathbf x_1) + \hat{\mathbf I}_j
  % (\mathbf x_2) ~\forall \mathbf y \in \bbR^d$. 
  Fix $t \in [0,1]$. Then the sequence of sample paths $\{(\tilde
  X^n(s),~s \in [0,t])\}$ satisfies the LDP with good rate function,
  for $\psi \in \sC[0,t]$,
  \[
  \hat J_t(\psi) = \inf_{\substack{\phi \in \bar \sC[0,t]\\\dot \phi(s) - \dot \psi(s) \geq
      0,~s \in [0,t]} }\hat\L_t(\phi) +
    \hat I_t(\phi - \psi).
  \]
\end{proposition}

\begin{proof}
 The independence of $(\tilde T^n(s),~s\in[0,t])$ and $(\tilde
 S^n(s),~s\in[0,t])$ for each $n \geq 1$ implies that they
  jointly satisfy the LDP with good rate function $\hat
  \Lambda_t(f) + \hat I_t(g)$ as a consequence of \cite[Corollary 2.9]{LySe1987},
  and where $(f,g) \in \bar \sC[0,t] \times \bar \sC[0,t]$. Since subtraction is
  continuous on the Polish space $\sC[0,t]$ equipped with the topology
  of pointwise convergence, applying the
  contraction principle along with Lemma~\ref{lem:increments-os} and \cite[Lemma 5.1.8]{DeZe2010}
  completes the proof.
$\QED$
\end{proof}
% Observe that the rate function can also be expressed as $\hat
% {J}_j(\mathbf y) = \inf_{\mathbf x \in [0,1]^d} \hat L_j(\mathbf x) +
% \hat{\mathbf I}_j(\mathbf x - \mathbf y)$ for all $\mathbf y \in
% \bbR^d$. This `unconstrained' expression will be useful for expressing
% the rate function of the offered load in the next result.

% \begin{proposition}
%   \label{prop:sample-path-offered}
%   Fix $t \in [0,1]$. Then, the sequence of sample paths $\{(\tilde
%   X^n(s), ~s\in [0,t]),~n \geq 1\}$
%   satisfies an LDP with good rate function
%   \begin{eqnarray}
%     \nonumber
%     \hat J_t(\phi) &=& \sup_{j \in \sJ_t} \left\{ \hat
%                        J_j(p_j(\phi)) \right\}~\forall \phi \in \sX\\
%     &=& \sup_{j \in \sJ_t} \left\{ \inf_{\mathbf x \in [0,1]^d} \hat \L_j(\mathbf x) +
% \hat{\mathbf I}_j(\mathbf x - \mathbf p_j(\phi)) \right\}
%   \end{eqnarray}
% \end{proposition}

% \begin{proof}
% There are two parts to the proof. First, we must argue that the
% supremum over the set $\sJ_t$ can be identified with $\sX$. Second, we
% must argue that the projective limit is the rate function. The former,
% however, follows from the proof of~\cite[Lemma 5.1.6]{DeZe2010}, and
% we will not repeat the arguments here. Second, the
% \end{proof}

As an illustration of the result, suppose that the service times are
exponentially distributed with mean 1. Define the function
\[
\begin{split}
\check  J_t(\phi,\psi) := \int_0^t \bigg( \log\left( \dot \phi^{(a)}(s) \right) &+ s \log \left(
    \frac{\dot \phi^{(a)}(s)- \dot \psi^{(a)}(s)}{s} \right) \bigg) ds\\ &-
\left(\phi^{(s)}(t) - \psi^{(s)}(t) - \frac{t^2}{2} + (1-t) \log \left(
  \frac{1-t}{1-\phi(t)}\right) \right).
\end{split}
\]
Then, the rate function
for the offered load sample path is
\begin{equation}
\label{eq:offered-exp-service}
\begin{split}
\hat J_t(\psi) = \inf_{\substack{\phi \in \bar \sC[0,t] \\ \dot \phi(t) - \dot
  \psi(s) \geq 0,~s\in [0,t]}} -\check J_t(\phi,\psi).
\end{split}
\end{equation}

%\subsection{Workload LDP}

We now establish the LDP for the workload process at a fixed $t \in [0,1]$.
\begin{theorem} \label{thm:ldp-workload}
  Fix $t\in[0,1]$. Then, the sequence of random variables $\{
  W^n(t),~n\geq 1\}$ satisfy the LDP with good rate function $\tilde
  J_t(y) = \inf_{\{\phi \in \sX : y = \Gamma(\phi)(t)\}} \hat
  J_t(\phi)$ for all $y \in \bbR$.
\end{theorem}

\begin{proof}
Recall that $\Gamma: \sC[0,t] \to \sC[0,t]$ is
continuous. Furthermore, $\sC[0,t]$ (under the topology of pointwise
convergence) and $\bbR$ are Hausdorff spaces. Therefore, the conditions of the
contraction principle \cite[Theorem 4.2.1]{DeZe2010} are
satisfied. Thus, it follows that $\{\tilde W^n(t),~n \geq 1\}$
satisfies the LDP with the rate function $\tilde J_t$. Finally, the
exponential equivalence proved in~Proposition
\ref{prop:exponential-eq} implies that $\{W^n(t),~n\geq 1\}$ 
satisfies the LDP with rate function $\tilde J_t$ , thus completing the proof.
$\QED$
\end{proof}

\subsection{Effective Bandwidths}
As noted in Section~\ref{sec:model}, our primary motivation for
studying the large deviation principle is to model the likelihood that
the workload at any point in time $t \in [0,1]$ exceeds a large
threshold. This is also related to the fact that most queueing models
in practice have finite-sized buffers, and so understanding the
likelihood that the workload exceeds the buffer capacity is crucial
from a system operation perspective. More precisely, if $w \in
[0,\infty)$ is the buffer capacity, we are interested
in probability of the event $\{W^n(t) > w\}$. Theorem~\ref{thm:ldp-workload} implies that
\[
\bbP\left( W^n(t) > w \right) \leq \exp(-n \tilde J_t((w,\infty))),
\]
where $\tilde J_t((w,\infty)) = \inf_{y \in (w,\infty)} \tilde
J_t(y)$. A reasonable performance measure to consider in this model is
to find the `critical time-scale' at which the large exceedance occurs
with probability at most $p$. That is, we would like to find
\[
t^* := \inf \{t > 0 | \exp(-n \tilde J_t((w,\infty))) \leq p\}.
\]
Consider the inequality $\tilde J_t((w,\infty)) \geq -\frac{1}{n} \log p$. Using the definition of rate function, we have
\[
\begin{split}
\inf_{f \in \sX : y = \Gamma(f)(t)} \inf_{\phi \in \bar
  \sC^0_f[0,t]} -\int_0^t \bigg( &\log(\dot
  \phi^{(a)}(r)) + \L^*(\dot \phi^{(a)}(r) - \dot f^{(a)}(r))  \bigg)
dr \\ &+ (1-t) \log \left( \frac{1-t}{1-\phi(t)} \right) + (\phi^{(s)}(t)
- f^{(s)}(t)) \geq -\frac{\log p}{n},
\end{split}
\]
where we define $\bar \sC^{0}_f[0,t] := \{g \in \bar \sC[0,t] : \dot g(s) - \dot
f(s) \geq 0, ~\forall s \in [0,t]\}$ for brevity. The critical
time-scale will be the optimizer of this constrained variational problem.
\section{Conclusions}
The large deviation principle derived for the `uniform scattering'
case in this paper provides the first result on the rare event
behavior of the $RS/GI/1$ transitory queue, building on the fluid and diffusion
approximation results established in
\cite{Lo94,HoJaWa2014,HoGl2016,BeHoJvL2015}. Our results are an
important addition to the body of knowledge dealing with rare events
behavior of queueing models. In particular, a standard assumption is that
the traffic model has independent increments, while our model assumes
exchangeable increments in the traffic count process by design. We
believe that the results in this paper are the first to report large
deviations analyses of queueing models under such conditions.

Our next step in this line of research will be to extend the analysis to
queues with non-uniform arrival epoch distributions, including
distributions that are not absolutely continuous. In this case, the
contraction principle cannot be directly applied, complicating the
analysis somewhat. In
\cite{HoGl2016} we have made initial progress under a `near-balanced'
condition on the offered load process, where the traffic and service
effort (on average) are approximately equal. However, it is unclear
how to drop the assumption of near-balancedness. In particular, when
the distribution is general, it is possible for the queue to enter
periods of underload, overload and critical load in the fluid
limit. This must have a significant impact on how the random variables
are `twisted' to rare outcomes. We do not believe it will be possible
to exploit ~\eqref{eq:order-expo} to establish the LDP. A
further problem of interest is to consider a different acceleration
regime. In the current setting we assumed that the service times
$\nu_i$ are scaled by the population size. However, it is possible to entertain
alternate scalings, such as $\nu_i^n = n^{-1}\nu_i(1+\beta
n^{1/3})$ as in \cite{BeHoJvL2015}, or scalings that are
dependent on the operational time horizon of interest. We leave these
problems to future papers.
%Recall that the upper bounding rate function obtained in Section \ref{sec:randomized} is $-\theta(t + w) + \log (1 - t + t \phi(\theta))$. It is straightforward to see that this function is upper bounded by $ -\theta(t+a) + t (\phi(\theta) - 1)$.

\acks
We thank Peter W. Glynn and Vijay Subramanian for many helpful discussions over the
course of writing this paper. We also thank the anonymous reviewer for
pointing out that it was possible to significantly extend the analysis
to prove a general LDP, and providing helpful hints on how to achieve
this. This research was supported in part by the National Science
Foundation through grant CMMI-1636069.
\bibliography{refs-queueing}
\bibliographystyle{apt}
\end{document}